\documentclass[12pt,reqno]{amsart}

\usepackage{amssymb,amsmath,graphicx,amsfonts,euscript}
\usepackage{color}

\newtheorem{thm}{Theorem}[section]

\newtheorem{prop}[thm]{Proposition}
\newtheorem{define}[thm]{Definition}

\newtheorem{lemma}[thm]{Lemma}

\newcommand{\p}{\partial}
\newcommand{\om}{\omega}

\newcommand{\Dd}{\Delta}

\numberwithin{equation}{section}

\begin{document}
\title[MHD equations with Magnetic Diffusion]
{The 2D Incompressible Magnetohydrodynamics Equations with only Magnetic Diffusion}

\author[C. Cao, J. Wu, B. Yuan]{Chongsheng Cao$^{1}$, Jiahong Wu$^{2}$, Baoquan Yuan$^{3}$}
\address{$^1$ Department of Mathematics,
Florida International University,
Miami, FL 33199, USA}

\email{caoc@fiu.edu}

\address{$^2$ Department of Mathematics,
Oklahoma State University,
401 Mathematical Sciences,
Stillwater, OK 74078, USA}

\email{jiahong@math.okstate.edu}

\address{$^3$ School of Mathematics and Information Science, Henan Polytechnic University,
Jiaozuo City, Henan  454000, P. R. China}

\email{bqyuan@hpu.edu.cn}

\vskip .2in
\begin{abstract}
This paper examines the global (in time) regularity of classical solutions to the 2D
incompressible magnetohydrodynamics (MHD) equations with only magnetic diffusion. Here the magnetic
diffusion is given by the fractional Laplacian operator $(-\Delta)^\beta$. We establish the
global regularity for the case when $\beta>1$. This result significantly improves previous work which
requires $\beta>\frac32$ and brings us closer to the resolution
of the well-known global regularity problem on the 2D MHD equations with standard Laplacian
magnetic diffusion, namely the case when $\beta=1$.
\end{abstract}

\maketitle

\section{Introduction}

This paper focuses on the initial-value problem (IVP) for the two-dimensional (2D) incompressible
magneto-hydrodynamics (MHD) equations with fractional Laplacian magnetic diffusion
\begin{equation} \label{MHDb}
\begin{cases}
\p_t u + u\cdot\nabla u  = -\nabla p + b\cdot\nabla b,  \quad x\in \mathbb{R}^2,\,t>0, \\
\p_t b + u\cdot\nabla b +  (-\Dd)^\beta b = b\cdot\nabla u,  \quad x\in \mathbb{R}^2,\,t>0, \\
\nabla \cdot u =0, \quad \nabla \cdot b =0, \quad x\in \mathbb{R}^2,\,t>0, \\
u(x,0) =u_0(x), \quad b(x,0) =b_0(x), \quad x\in \mathbb{R}^2,
\end{cases}
\end{equation}
where the fractional Laplacian operator $(-\Delta)^\beta$ is defined through the Fourier transform
$$
\widehat{(-\Delta)^\beta \,f} (\xi) = |\xi|^{2\beta} \, \widehat{f}(\xi)
$$
with $\widehat{f}$ being the Fourier transform of $f$, namely
$$
\widehat{f}(\xi) = \int_{\mathbb{R}^2}  e^{-i x\cdot \xi} \, f(x)\, dx.
$$
When $\beta=1$, (\ref{MHDb}) reduces to the MHD equations with Laplacian magnetic diffusion, which models
many significant phenomena such as the magnetic reconnection in astrophysics and geomagnetic dynamo in geophysics
(see, e.g., \cite{Pri}).

\vskip .1in
What we care about here is the global regularity problem, namely whether the solutions of (\ref{MHDb})
emanating from sufficiently smooth data $(u_0, b_0)$ remain regular for all time. The main result of this
paper states that (\ref{MHDb}) with any $\beta>1$ always possesses a unique global solution. More precisely,
we have the following theorem.
\begin{thm} \label{main}
Consider (\ref{MHDb}) with $\beta>1$. Assume that $(u_0, b_0)\in H^s(\mathbb{R}^2)$ with $s>2$,
$\nabla\cdot u_0=0$, $\nabla\cdot b_0=0$ and
$j_0=\nabla \times b_0$ satisfying
$$
\|\nabla j_0\|_{L^\infty} <\infty.
$$
Then (\ref{MHDb}) has a unique global solution $(u, b)$ satisfying, for any $T>0$,
$$
(u,b) \in L^\infty([0,T]; H^s(\mathbb{R}^2)), \quad \nabla j \in L^1([0,T]; L^\infty(\mathbb{R}^2))
$$
where $j=\nabla\times b$.
\end{thm}

\vskip .1in
We now review some recent work to put our result in a proper content. Due to their prominent roles
in modeling many phenomena in astrophysics, geophysics and plasma physics, the MHD equations have been
studied extensively mathematically. G. Duvaut and J.-L. Lions constructed a global Leray-Hopf
weak solution and a local strong solution of the 3D incompressible MHD equations \cite{DL}. M. Sermange and
R. Temam further examined the properties of these solutions \cite{SeTe}. More recent work on the MHD
equations develops regularity criteria in terms of the velocity field and deals with
the MHD equations with dissipation and magnetic diffusion given by general Fourier multiplier operators
such as the fractional Laplacian operators (see, e.g.,
\cite{CMZ1,CMZ2,HeXin1,HeXin2,JiuZhao,TrYu,Wang,Wu2,Wu3,Wu4,YuanBai,Yama1,Yama2}).

\vskip .1in
Another direction that has generated considerable
interest recently is the global regularity problem on the MHD equations with partial dissipation, especially in
the 2D case (see, e.g., \cite{CaoWu,CaoReWu,Chae,JiuNiu,LZ,LinZhang1,LinZhang2,XuZhang,Zhou}).
Since the global regularity of the 2D MHD equations
with both Laplacian dissipation and magnetic diffusion is easy to obtain while the regularity of the
completely inviscid MHD equations appear to be out of reach, it is natural to examine the MHD equations
with partial dissipation. One particular partial dissipation case is (\ref{MHDb}). When $\beta\ge 1$, any solution
of (\ref{MHDb}) with $(u_0, b_0) \in H^1$ generates a global weak solution $(u,b)$ that remains bound in $H^1$
for all time (see, e.g., \cite{CaoWu,LZ}). However, it is not clear if such weak solutions are unique or their regularity
can be improved to be in $H^2$ if the initial data $(u_0, b_0)$ is in $H^2$. The result of this paper obtains
the uniqueness and global regularity for the case when $\beta>1$. Previous global regularity results
require that $\beta>\frac32$ (see, e.g., \cite{JiuZhao,TrYu,YuanBai,Yama1}).  Our approach here does not appear to be able to
extend to the borderline case when $\beta=1$, which is currently being studied by a different method \cite{Niu}.
Progress has also been made on several other partial dissipation cases of the MHD equations. F. Lin, L. Xu
and P. Zhang recently studied the MHD equations with the Laplacian dissipation in the velocity equation but without
magnetic diffusion and were able to establish the global existence of small solutions after translating
the magnetic field by a constant vector (\cite{LinZhang1,LinZhang2,XuZhang}). C. Cao and J. Wu examined the anisotrophic 2D MHD
equations with horizontal dissipation and vertical magnetic diffusion (or vertical dissipation and horizontal
magnetic diffusion) and obtained the global regularity for this partial dissipation case \cite{CaoWu}. The anisotrophic
MHD equations with horizontal dissipation and horizontal magnetic diffusion were also investigated very recently
and progress was also made (see \cite{CaoReWu}).

\vskip .1in
We now explain our proof of Theorem \ref{main}. Since the local (in time) well-posedness can be established following
a standard approach, our efforts are devoted to proving global {\it a priori} bounds for $(u,b)$ in the initial
functional setting $H^s$ with $s>2$. $(u,b)$ indeed admits a global $H^1$-bound, but direct energy estimates
do not appear to easily yield a global $H^2$-bound. The difficulty comes from the nonlinear term in the velocity
equation due to the lack of dissipation. To bypass this difficulty, we first make use of the magnetic diffusion
$(-\Delta)^\beta b$ with $\beta>1$ to show the global bound for
$\|\om\|_{L^q}$ and  $\|j\|_{L^q}$ for $2\le q\le \frac{2}{2-\beta}$ (the range is modified to $2\le q<\infty$ when $\beta=2$).
The magnetic diffusion is further exploited to obtain a global bound for $j$ in the space-time
Besov space $L^1_t B^s_{q,1}$, namely, for any $T>0$ and $t\le T$
\begin{equation}\label{jbesov}
\|j\|_{L^1_t B^s_{q,1}} \le C(T, u_0, b_0) <\infty \quad\mbox{with}\quad 2\le q\le \frac{2}{2-\beta},\quad \frac2q<s<2\beta-1.
\end{equation}
Roughly speaking, this global bound provides the time integrability of the $L^q$-norm of $j$ up to
$(2\beta-1)$-derivative. This global bound is proven through Besov space techniques. Special consequences
of this global bound include the time integrability of $\|j\|_{L^\infty}$ and of $\|\nabla j\|_{L^r}$ for $r>q$.
To gain higher regularity, we go through an iterative process. The bound $\|\nabla j\|_{L^1_t L^r} <\infty$ allows us
to further prove a global bound for $\|\om\|_{L^r}$ and $\|j\|_{L^r}$ with $r>q$, which can be employed
to prove (\ref{jbesov}) with $q$ replaced by $r$. Repeating this process leads to the global bound in (\ref{jbesov})
for any $q\in [2,\infty)$, which especially implies, for any $t>0$,
\begin{equation}\label{naj}
\int_0^t \|\nabla j\|_{L^\infty} d\tau <\infty.
\end{equation}
(\ref{naj})  can then be further used to obtain a global bound for $\|\om\|_{L^\infty}$. The time integrability of $\|j\|_{L^\infty}$
and the boundedness of $\|\om\|_{L^\infty_{t,x}}$ are enough to prove a global bound for $(u,b)$ in $H^s$.

\vskip .1in
The rest of this paper is divided into three sections. The second section provides the definition of inhomogeneous
Besov spaces and related facts such as Bernstein's inequality. The third section proves the global $L^q$-bound for
$(\omega, j)$ while the fourth section establishes the global bound in (\ref{jbesov}). The last section gains
higher regularity through an iterative process and proves Theorem \ref{main}.

\vskip .3in
\section{Functional spaces}

This section provides the definition of Besov spaces and related facts used in the subsequent sections.
Materials presented here can be found in several books and many papers
(see, e.g., \cite{BCD,BL,MWZ,RS,Tri}).

\vskip .1in
We start with several notation. $\mathcal{S}$ denotes
the usual Schwarz class and ${\mathcal S}'$ its dual, the space of
tempered distributions. It is a simple fact in analysis that there exist
two radially symmetric functions $\Psi, \Phi\in \mathcal{S}$
such that
$$
\mbox{supp} \widehat{\Psi} \subset B\left(0, \frac43\right),
\qquad \mbox{supp} \widehat{\Phi} \subset
\mathcal{A}\left(0, \frac34, \frac83\right),
$$
$$
\widehat{\Psi} (\xi) +  \sum_{j=0}^\infty \widehat{\Phi}_j (\xi) =1,
\quad \xi \in {\mathbb R}^d,
$$
where $B(0,r)$ denotes the ball centered at the origin with radius $r$, $\mathcal{A}(0,r_1,r_2)$ denotes the
annulus centered at the origin with the inner radius $r_1$ and the outer $r_2$,
$\Phi_0=\Phi$ and $\Phi_j (x) =2^{jd} \Phi_0(2^j x)$ or $\widehat{\Phi}_j(\xi) = \widehat{\Phi}_0(2^{-j} \xi)$.

\vskip .1in
Then, for any $\psi\in {\mathcal S}$,
$$
\Psi \ast \psi + \sum_{j=0}^\infty \Phi_j \ast \psi =\psi
$$
and hence
\begin{equation}\label{sf}
\Psi \ast f + \sum_{j=0}^\infty \Phi_j \ast f =f
\end{equation}
in ${\mathcal S}'$ for any $f\in {\mathcal S}'$. To define the inhomogeneous Besov space, we set
\begin{equation} \label{del2}
\Delta_j f = \left\{
\begin{array}{ll}
0,&\quad \mbox{if}\,\,j\le -2, \\
\Psi\ast f,&\quad \mbox{if}\,\,j=-1, \\
\Phi_j \ast f, &\quad \mbox{if} \,\,j=0,1,2,\cdots.
\end{array}
\right.
\end{equation}
\begin{define}
The inhomogeneous Besov space $B^s_{p,q}$ with $1\le p,q \le \infty$
and $s\in {\mathbb R}$ consists of $f\in {\mathcal S}'$
satisfying
$$
\|f\|_{B^s_{p,q}} \equiv \|2^{js} \|\Delta_j f\|_{L^p} \|_{l^q}
<\infty.
$$
\end{define}

\vskip .1in
Besides the Fourier localization operators $\Delta_j$, the partial sum $S_j$ is also a useful notation. For an integer $j$,
$$
S_j \equiv \sum_{k=-1}^{j-1} \Delta_k,
$$
where $\Delta_k$ is given by (\ref{del2}). For any $f\in \mathcal{S}'$, the Fourier transform of $S_j f$ is supported on the ball of radius $2^j$.

\vskip .1in
Bernstein's inequalities are useful tools in dealing with Fourier localized functions and these inequalities trade integrability for derivatives. The following proposition provides Bernstein type inequalities for fractional derivatives.
\begin{prop}\label{bern}
Let $\alpha\ge0$. Let $1\le p\le q\le \infty$.
\begin{enumerate}
\item[1)] If $f$ satisfies
$$
\mbox{supp}\, \widehat{f} \subset \{\xi\in \mathbb{R}^d: \,\, |\xi|
\le K 2^j \},
$$
for some integer $j$ and a constant $K>0$, then
$$
\|(-\Delta)^\alpha f\|_{L^q(\mathbb{R}^d)} \le C_1\, 2^{2\alpha j +
j d(\frac{1}{p}-\frac{1}{q})} \|f\|_{L^p(\mathbb{R}^d)}.
$$
\item[2)] If $f$ satisfies
\begin{equation*}\label{spp}
\mbox{supp}\, \widehat{f} \subset \{\xi\in \mathbb{R}^d: \,\, K_12^j
\le |\xi| \le K_2 2^j \}
\end{equation*}
for some integer $j$ and constants $0<K_1\le K_2$, then
$$
C_1\, 2^{2\alpha j} \|f\|_{L^q(\mathbb{R}^d)} \le \|(-\Delta)^\alpha
f\|_{L^q(\mathbb{R}^d)} \le C_2\, 2^{2\alpha j +
j d(\frac{1}{p}-\frac{1}{q})} \|f\|_{L^p(\mathbb{R}^d)},
$$
where $C_1$ and $C_2$ are constants depending on $\alpha,p$ and $q$
only.
\end{enumerate}
\end{prop}

\vskip .3in
\section{Global $L^q$-bound for $(\omega, j)$}

This section proves a global {\it a priori} bound for $\om$ and $j$ in the Lebesgue space $L^q(\mathbb{R}^2)$
for $q$ in a suitable range. More precisely, we prove the following proposition.
\begin{prop} \label{Lq}
Assume $(u_0, b_0)$ satisfies the conditions stated in Theorem \ref{main}. Let $(u, b)$ be the
corresponding solution of (\ref{MHDb}) with $\beta>1$.  Then, for any $q$ satisfying
\begin{equation}\label{qran}
2 \le q \le \frac{2}{2-\beta}
\end{equation}
(the range of $q$ is modified to $2\le q<\infty$ when $\beta=2$), and for any $T>0$ and $t\le T$, there exists a constant
$C=C(T, u_0, b_0)$ such that
\begin{equation}\label{Lqbd}
\|\om(t)\|_{L^q} \le C,\qquad \|j(t)\|_{L^q} \le C.
\end{equation}
\end{prop}

\vskip .1in
To prove Proposition \ref{Lq}, we first provide two simple bounds. A simple energy estimate yields the
global $L^2$-bound of (\ref{MHDb}) with $\beta\ge 0$.
\begin{lemma} \label{L2}
Assume $(u_0, b_0)$ satisfies the conditions stated in Theorem \ref{main}. Let $(u, b)$ be the
corresponding solution of (\ref{MHDb}) with $\beta \ge 0$. Then, for any $t\ge 0$,
$$
\|u(t)\|_{L^2}^2 + \|b(t)\|_{L^2}^2 + 2 \int_0^t \|\Lambda^\beta b(\tau)\|_{L^2}^2\,d\tau
= \|u_0\|_{L^2}^2 + \|b_0\|_{L^2}^2,
$$
where $\Lambda=\sqrt{-\Delta}$.
\end{lemma}

In addition, by resorting to the equations of $\om$ and $j$,
\begin{equation}\label{vj}
\begin{cases}
\p_t \om + u\cdot\nabla \om  = b\cdot\nabla j,  \\
\p_t j + u\cdot\nabla j +  (-\Dd)^\beta j = b\cdot\nabla \om + 2 \partial_1b_1(\p_2 u_1 + \p_1 u_2) - 2 \partial_1u_1(\p_2 b_1 + \p_1 b_2), \\
\om(x,0) =\om_0(x), \quad \om(x,0) =j_0(x),
\end{cases}
\end{equation}
the global $H^1$-bound on $(u,b)$ can be obtained in a similar fashion as in \cite{CaoWu}.
\begin{lemma} \label{H1}
Assume $(u_0, b_0)$ satisfies the conditions stated in Theorem \ref{main}. Let $(u, b)$ be the
corresponding solution of (\ref{MHDb}) with $\beta \ge 1$. Then, for any $T>0$ and $t\le T$, there exists a constant
$C=C(T, \|(u_0, b_0)\|_{H^1})$ such that
\begin{equation} \label{H1int}
\|\om(t)\|_{L^2}^2 + \|j(t)\|_{L^2}^2 + \int_0^t \|\Lambda^\beta j(\tau)\|_{L^2}^2\,d\tau
\le C.
\end{equation}
\end{lemma}

\vskip .1in
\begin{proof}[Proof of Proposition \ref{Lq}]
Multiplying the first equation in (\ref{vj}) by $\om |\om|^{q-2}$, integrating in space and applying H\"{o}lder's inequality, we have
$$
\frac1q \frac{d}{dt} \|\om\|_{L^q}^q = \int b\cdot \nabla j\, \om |\om|^{q-2}
\le \|\om\|_{L^q}^{q-1}\, \|b\|_{L^\infty} \|\nabla j\|_{L^q}.
$$
Recall the simple Sobolev inequalities, for $\beta>1$,
\begin{equation}\label{bL8}
\|b\|_{L^\infty} \le C \|b\|_{L^2}^{1-\frac1{1+\beta}}\, \|\Lambda^\beta j\|_{L^2}^{\frac1{1+\beta}}, \qquad
\|\nabla j\|_{L^q} \le C\, \|j\|_{L^2}^{1-\frac{2(q-1)}{\beta q}}\,\|\Lambda^\beta j\|_{L^2}^{\frac{2(q-1)}{\beta q}}
\end{equation}
and note that (\ref{qran}) ensures that $\frac{2(q-1)}{\beta q} \le 1$. By Young's inequality,
$$
\|b\|_{L^\infty} \|\nabla j\|_{L^q} \le C\,(\|b\|_{L^2}^2 + \|j\|_{L^2}^2 + \|\Lambda^\beta j\|_{L^2}^2),
$$
Therefore,
$$
\|\om(t)\|_{L^q} \le \|\om_0\|_{L^q} + C\,\int_0^t (\|b\|_{L^2}^2 + \|j\|_{L^2}^2 + \|\Lambda^\beta j\|_{L^2}^2)\,dt
$$
and the bounds in Lemmas \ref{L2} and \ref{H1} yield the global bound for $\|\om(t)\|_{L^q}$. To get a global bound for
$\|j\|_{L^q}$, we first obtain from the equation of $j$ in (\ref{vj})
\begin{equation}\label{Lqjroot}
\frac1q \frac{d}{dt} \|j\|_{L^q}^q + \int j|j|^{q-2}\, (-\Delta)^\beta j =K_1 + K_2 + K_3 +K_4 +K_5,
\end{equation}
where $K_1, \cdots, K_5$ are given by
\begin{eqnarray*}
K_1 &=& \int b\cdot \nabla \om \, j|j|^{q-2}, \\
K_2 &=& 2 \int \partial_1b_1\,\p_2 u_1 \, j|j|^{q-2},\\
K_3 &=& 2 \int \partial_1b_1\,\p_1 u_2 \, j|j|^{q-2},\\
K_4 &=& 2 \int \partial_1u_1\,\p_2 b_1 \, j|j|^{q-2},\\
K_5 &=& 2 \int \partial_1u_1\, \p_1 b_2 \, j|j|^{q-2}.\\
\end{eqnarray*}
According to \cite{CC}, we have the lower bound, for $C_0=C_0(\beta,q)$,
\begin{equation}\label{Lower}
\int j|j|^{q-2}\, (-\Delta)^\beta j \ge  C_0\,\int |\Lambda^\beta (|j|^{\frac{q}{2}})|^2.
\end{equation}
By integration by parts and H\"{o}lder's inequality,
\begin{eqnarray*}
|K_1| &=& \frac{2(q-1)}{q} \left|\int \om \,b\cdot \nabla(|j|^{\frac{q}{2}})\, |j|^{\frac{q}{2}-1} \right| \\
&\le & C\, \|b\|_{L^\infty} \, \|\om\|_{L^q} \, \|j\|_{L^q}^{\frac{q}{2}-1} \, \|\nabla(|j|^{\frac{q}{2}})\|_{L^2}.
\end{eqnarray*}
By the trivial embedding inequality, for $\beta \ge 1$,
$$
\|\nabla(|j|^{\frac{q}{2}})\|_{L^2} \le \||j|^{\frac{q}{2}}\|_{H^\beta} \le \||j|^{\frac{q}{2}}\|_{L^2} +
\|\Lambda^\beta(|j|^{\frac{q}{2}})\|_{L^2} =\|j\|_{L^q}^{\frac{q}{2}} + \|\Lambda^\beta(|j|^{\frac{q}{2}})\|_{L^2},
$$
we obtain
\begin{eqnarray*}
|K_1| &\le& \frac{C_0}{2} \int |\Lambda^\beta (|j|^{\frac{q}{2}})|^2 + C\,\|b\|_{L^\infty} \, \|\om\|_{L^q}
\,\|j\|_{L^q}^{q-1} + C\, \|b\|^2_{L^\infty} \, \|\om\|^2_{L^q}
\,\|j\|_{L^q}^{q-2}\\
&\le& \frac{C_0}{2} \int |\Lambda^\beta (|j|^{\frac{q}{2}})|^2 + C\,(1+\|b\|_{L^\infty})\|b\|_{L^\infty}
(\|\om\|_{L^q}^{q} + \|j\|_{L^q}^{q}).
\end{eqnarray*}
$K_2$ can be easily bounded. In fact, by the Sobolev inequality with $\beta>1$,
$$
\|f\|_{L^\infty(\mathbb{R}^2)} \le C\, \|f\|_{L^2(\mathbb{R}^2)}^{1-\frac1\beta}\,
\|\Lambda^\beta f\|_{L^2(\mathbb{R}^2)}^{\frac1\beta},
$$
we have
\begin{eqnarray}
|K_2| &\le& 2 \|\partial_1b_1\|_{L^\infty}\, \|\om\|_{L^q} \,\|j\|_{L^q}^{q-1} \nonumber\\
&\le&  C\, \|\partial_1b_1\|_{L^2}^{1-\frac1\beta}\, \|\Lambda^\beta\partial_1b_1\|_{L^2}^{\frac1\beta}\,
(\|\om\|_{L^q}^q + \|j\|_{L^q}^q) \nonumber\\
&\le&  C\, \|j\|_{L^2}^{1-\frac1\beta}\, \|\Lambda^\beta j\|_{L^2}^{\frac1\beta}\,
(\|\om\|_{L^q}^q + \|j\|_{L^q}^q), \label{k2b}
\end{eqnarray}
where the simple inequality $\|\partial_1b_1\|_{L^2} \le \|\nabla b\|_{L^2} =\|\nabla\times b\|_{L^2} =\|j\|_{L^2}$.
Clearly, $K_3$, $K_4$ and $K_5$ admits the same bound as in (\ref{k2b}). Collecting the estimates, we have
\begin{eqnarray*}
&&\frac1q \frac{d}{dt} \|j\|_{L^q}^q + \frac{C_0}{2} \int |\Lambda^\beta (|j|^{\frac{q}{2}})|^2 \\
&& \qquad\qquad \le C\,(\|b\|_{L^\infty}+\|b\|_{L^\infty}^2+ \|j\|_{L^2}^{1-\frac1\beta}\, \|\Lambda^\beta j\|_{L^2}^{\frac1\beta})
(\|\om\|_{L^q}^{q} + \|j\|_{L^q}^{q}).
\end{eqnarray*}
Thanks to the time integrability of  $\|b\|_{L^\infty}^2$ (see (\ref{bL8}) for a bound) and
$\|j\|_{L^2}^{1-\frac1\beta}\, \|\Lambda^\beta j\|_{L^2}^{\frac1\beta}$
on any finite time interval and the global bound for $\|\om\|_{L^q}$, this differential
inequality yields the global bound for $\|j\|_{L^q}$. This completes the proof of Proposition \ref{Lq}.
\end{proof}

\vskip .3in
\section{Global $L^1_t B^s_{q,1}$-bound for $j$}

\vskip .1in
This section establishes a global bound for $j$ in the space-time Besov space $L^1_t B^s_{q,1}$, where $q$ satisfies
(\ref{qran}) and $\frac2q <s <2\beta-1$. For $\beta>1$, this global bound provides a better integrability than the one
given in (\ref{H1int}) and will be exploited to gain higher regularity in the next section.

\vskip .1in
\begin{prop} \label{Bsq1}
Assume $(u_0, b_0)$ satisfies the conditions stated in Theorem \ref{main}. Let $(u, b)$ be the
corresponding solution of (\ref{MHDb}) with $\beta>1$.  Then, for any $q$ and $s$ satisfying
\begin{equation}\label{qscon}
2 \le q \le \frac{2}{2-\beta}, \qquad \frac2q <s <2\beta-1
\end{equation}
(the range of $q$ is modified to $2\le q<\infty$ when $\beta=2$), and for any $T>0$ and $t\le T$, there exists a constant
$C=C(q, s, T, u_0, b_0)$ such that
\begin{equation}\label{jBesov}
\|j\|_{L^1_t B^s_{q,1}} \le C.
\end{equation}
\end{prop}

We remark that the proof of this theorem makes use of the global $L^q$-bound for $\|\om\|_{L^q}$ and $\|j\|_{L^q}$
and this explains why we need to restrict $q$ to the range in (\ref{qscon}). It can be seen from the proof
that this theorem remains valid for any $q\ge 2$ as long as  $\|\om\|_{L^q}$ and $\|j\|_{L^q}$ are bounded.
We now prove Proposition \ref{Bsq1}.

\begin{proof}[Proof of Proposition \ref{Bsq1}]
Let $k\ge 0$ be an integer. Applying $\Delta_k$ to the equation of $j$ in (\ref{vj}), multiplying by
$\Delta_k j\, |\Delta_k j|^{q-2}$ and integrating in space, we obtain
\begin{equation} \label{deltaroot}
\frac1q \frac{d}{dt} \|\Delta_k j\|_{L^q}^q + \int \Delta_k j\, |\Delta_k j|^{k-2} (-\Delta)^\beta \Delta_k j = L_1 + \cdots+ L_6,
\end{equation}
where $L_1$,$\cdots$, $L_6$ are given by
\begin{eqnarray*}
L_1 &=& -\int \Delta_k j\, |\Delta_k j|^{q-2}\, \Delta_k(u\cdot\nabla j),\\
L_2 &=& \int\Delta_k j\, |\Delta_k j|^{q-2}\, \Delta_k(b\cdot\nabla \om),\\
L_3 &=& 2\int\Delta_k j\, |\Delta_k j|^{q-2}\, \Delta_k(\partial_1 b_1 \partial_2 u_1),\\
L_4 &=& 2\int\Delta_k j\, |\Delta_k j|^{q-2}\, \Delta_k(\partial_1 b_1 \partial_1 u_2),\\
L_5 &=& -2\int\Delta_k j\, |\Delta_k j|^{q-2}\, \Delta_k(\partial_1 u_1 \partial_2 b_1),\\
L_6 &=& -2\int\Delta_k j\, |\Delta_k j|^{q-2}\, \Delta_k(\partial_1 u_1 \partial_1 b_2).
\end{eqnarray*}
The term associated the magnetic diffusion admits the lower bound (see \cite{CMZ})
$$
\int \Delta_k j\, |\Delta_k j|^{q-2} (-\Delta)^\beta \Delta_k j \ge C_1 2^{2\beta k} \|\Delta_k j\|_{L^q}^q
$$
for $C_1=C_1(\beta, q)$. By Bony's paraproducts decomposition, we write
\begin{eqnarray*}
L_1 =L_{11} + L_{12} + L_{13} + L_{14} + L_{15},
\end{eqnarray*}
where
\begin{eqnarray*}
L_{11} &=& -\int \Delta_k j\, |\Delta_k j|^{q-2}\, \sum_{|k-l|\le 2} [\Delta_k, S_{l-1} u\cdot\nabla] \Delta_l j,\\
L_{12} &=& -\int \Delta_k j\, |\Delta_k j|^{q-2}\, \sum_{|k-l|\le 2} (S_{l-1} u-S_k u)\cdot\nabla \Delta_k\Delta_l j, \\
L_{13} &=& -\int \Delta_k j\, |\Delta_k j|^{q-2}\, S_k u \cdot\nabla \Delta_k j, \\
L_{14} &=& -\int \Delta_k j\, |\Delta_k j|^{q-2}\, \sum_{|k-l|\le 2} \Delta_k(\Delta_l u\cdot\nabla S_{l-1} j), \\
L_{15} &=& -\int \Delta_k j\, |\Delta_k j|^{q-2}\, \sum_{l\ge k-1} \Delta_k(\Delta_l u\cdot\nabla \widetilde{\Delta}_l j)
\end{eqnarray*}
with $\widetilde{\Delta}_l =\Delta_{l-1} + \Delta_l + \Delta_{l+1}$. Thanks to $\nabla\cdot S_k u =0$, we have $L_{13}=0$.
By H\"{o}lder's inequality, a standard commutator estimate, and Bernstein's inequality,
\begin{eqnarray}
|L_{11}| &\le& \|\Delta_k j\|_{L^q}^{q-1}\, \sum_{|k-l|\le 2} \|[\Delta_k, S_{l-1} u\cdot\nabla] \Delta_l j\|_{L^q} \nonumber \\
 &\le& C\, \|\Delta_k j\|_{L^q}^{q-1}\, \|\nabla S_{k-1} u\|_{L^q} \, \|\Delta_k j\|_{L^\infty} \nonumber \\
 &\le& C\, \|\Delta_k j\|_{L^q}^{q-1}\, \|\om\|_{L^q}\, 2^{k\frac2q} \|\Delta_k j\|_{L^q}. \label{l11b}
\end{eqnarray}
Here we have also applied the simple fact that, for fixed $k$, the summation is for a finite number of $l$ satisfying
$|k-l|\le 2$ and the estimate for the term with the index $l$ is only a constant multiple of the bound for the term
with the index $k$. In addition, the simple bound
$$
\|\nabla S_{k-1} u\|_{L^q}  \le \|\nabla u\|_{L^q}  \le C \|\om\|_{L^q}
$$
is also used here. It is easily seen that $L_{12}$ obeys the same bound as in (\ref{l11b}). By H\"{o}lder's inequality
and and Bernstein's inequality (both lower bound and upper bound parts),
\begin{eqnarray*}
|L_{14}| &\le&  C\, \|\Delta_k j\|_{L^q}^{q-1}\, \|\Delta_k u\|_{L^q} \,\|\nabla S_{k-1} j\|_{L^\infty}\\
&\le&  C\, \|\Delta_k j\|_{L^q}^{q-1}\, 2^{-k}\, \|\nabla \Delta_k u\|_{L^q} \,
\sum_{m\le k-1} 2^{(1+\frac{2}{q})m} \|\Delta_m j\|_{L^q} \\
&\le&  C\, \|\om\|_{L^q}\,\|\Delta_k j\|_{L^q}^{q-1}\, 2^{-k}\,  \sum_{m\le k-1} 2^{(1+\frac{2}{q})m} \|\Delta_m j\|_{L^q}.
\end{eqnarray*}
By the divergence-free condition, $\nabla\cdot \Delta_l u=0$,
\begin{eqnarray*}
|L_{15}| &\le&  C\, \|\Delta_k j\|_{L^q}^{q-1}\, \sum_{l\ge k-1} 2^k \|\Delta_l u\|_{L^q} \,\|\widetilde{\Delta}_l j\|_{L^\infty}\\
 &\le&  C\, \|\Delta_k j\|_{L^q}^{q-1}\, \sum_{l\ge k-1} 2^{k-l} \|\nabla \Delta_l u\|_{L^q}\, 2^{l\frac2q}\, \|\widetilde{\Delta}_l j\|_{L^q}\\
 &\le&  C\,\|\om\|_{L^q}\,\|\Delta_k j\|_{L^q}^{q-1}\, 2^{k\frac2q}\, \sum_{l\ge k-1} 2^{(k-l)(1-\frac2q)}\, \|\Delta_l j\|_{L^q}.
\end{eqnarray*}
We have thus completed the estimate for $L_1$. To bound $L_2$, we also decompose it into five terms,
\begin{eqnarray*}
L_2 =L_{21} + L_{22} + L_{23} + L_{24} + L_{25},
\end{eqnarray*}
where
\begin{eqnarray*}
L_{21} &=& -\int \Delta_k j\, |\Delta_k j|^{q-2}\, \sum_{|k-l|\le 2} [\Delta_k, S_{l-1} b\cdot\nabla] \Delta_l \om,\\
L_{22} &=& -\int \Delta_k j\, |\Delta_k j|^{q-2}\, \sum_{|k-l|\le 2} (S_{l-1} b-S_k b)\cdot\nabla \Delta_k \Delta_l \om, \\
L_{23} &=& -\int \Delta_k j\, |\Delta_k j|^{q-2}\, S_k b \cdot\nabla \Delta_k \om, \\
L_{24} &=& -\int \Delta_k j\, |\Delta_k j|^{q-2}\, \sum_{|k-l|\le 2} \Delta_k(\Delta_l b\cdot\nabla S_{l-1} \om), \\
L_{25} &=& -\int \Delta_k j\, |\Delta_k j|^{q-2}\, \sum_{l\ge k-1} \Delta_k(\Delta_l b\cdot\nabla \widetilde{\Delta}_l \om).
\end{eqnarray*}
The estimates of these terms are similar to those in $L_1$, but there are some differences, as can be seen from the bounds below.
\begin{eqnarray*}
|L_2| &\le&  C\, 2^{k\frac2q}\,\|\Delta_k j\|_{L^q}^{q-1}\, \|\nabla b\|_{L^q}\,\|\om\|_{L^q}
+ C\, 2^{k}\,\|\Delta_k j\|_{L^q}^{q-1}\, \|b\|_{L^\infty}\,\|\om\|_{L^q} \\
&& + C\, 2^{k\frac2q}\,\|\Delta_k j\|_{L^q}^{q-1}\,\sum_{l\ge k-1} 2^{(k-l)(1-\frac{2}{q})} \|\nabla \Delta_l b\|_{L^q}\,
\|\Delta_l \om\|_{L^q}\\
 &\le&  C\, 2^{k\frac2q}\,\|\Delta_k j\|_{L^q}^{q-1}\, \|\nabla b\|_{L^q}\,\|\om\|_{L^q}
+ C\, 2^{k}\,\|\Delta_k j\|_{L^q}^{q-1}\, \|b\|_{L^\infty}\,\|\om\|_{L^q}.
\end{eqnarray*}
To bound $L_3$, we decompose it into three terms as
\begin{eqnarray*}
L_3 =L_{31} + L_{32} + L_{33},
\end{eqnarray*}
where
\begin{eqnarray*}
L_{31} &=& 2\int \Delta_k j\, |\Delta_k j|^{q-2}\, \sum_{|k-l|\le 2}\Delta_k(S_{l-1} \partial_1 b_1 \Delta_l\partial_2 u_1), \\
L_{32} &=& 2\int \Delta_k j\, |\Delta_k j|^{q-2}\, \sum_{|k-l|\le 2}\Delta_k(\Delta_l \partial_1 b_1 S_{l-1}\partial_2 u_1), \\
L_{33} &=& 2\int \Delta_k j\, |\Delta_k j|^{q-2}\, \sum_{l\ge k-1} \Delta_k(\Delta_l \partial_1 b_1\widetilde{\Delta}_l\partial_2 u_1).
\end{eqnarray*}
These terms can be bounded by
\begin{eqnarray*}
|L_{31}| &\le& C\,\|\Delta_k j\|_{L^q}^{q-1}\, \|S_{k-1} \partial_1 b_1\|_{L^\infty} \, \|\Delta_k\partial_2 u_1\|_{L^q} \\
&\le& C\,\|\Delta_k j\|_{L^q}^{q-1}\, 2^{k\frac2q}\,\|j\|_{L^q}\,\|\om\|_{L^q}
\end{eqnarray*}
where we have used the bound $\|\nabla u\|_{L^q} \le C\, \|\om\|_{L^q}$ and $\|\nabla b\|_{L^q} \le C\, \|j\|_{L^q}$.  Clearly,
$L_{32}$ admits the same bound. $L_{33}$ is bounded by
\begin{eqnarray*}
|L_{33}| &\le& C\, \|\Delta_k j\|_{L^q}^{q-1}\, \sum_{l\ge k-1} 2^{l\frac2q}\,\|\Delta_l \partial_1 b_1\|_{L^q}\,
\|\Delta_l\partial_2 u_1\|_{L^q} \\
&\le& C\, \|\om\|_{L^q}\, \|\Delta_k j\|_{L^q}^{q-1}\, \sum_{l\ge k-1} 2^{l\frac2q}\,\|\Delta_l j\|_{L^q}.
\end{eqnarray*}
Collecting all the estimates above, we obtain
\begin{eqnarray*}
\frac{d}{dt} \|\Delta_k j\|_{L^q} + C_1\, 2^{2\beta k} \|\Delta_k j\|_{L^q} \le RHS(t),
\end{eqnarray*}
where $RHS(t)$ denotes the bound
\begin{eqnarray}
&& RHS(t)\equiv C\,\|\om\|_{L^q}\, 2^{k\frac2q} \|\Delta_k j\|_{L^q} \nonumber\\
&&\qquad \qquad + \, C\, \|\om\|_{L^q}\, 2^{-k}\,  \sum_{m\le k-1} 2^{(1+\frac{2}{q})m} \|\Delta_m j\|_{L^q} \nonumber\\
&&\qquad \qquad + \, C\, 2^{k\frac2q}\,\|\om\|_{L^q}\, \sum_{l\ge k-1} 2^{(k-l)(1-\frac2q)}\, \|\Delta_l j\|_{L^q} \nonumber\\
&&\qquad \qquad + \,C\, 2^{k\frac2q}\,\|j\|_{L^q}\,\|\om\|_{L^q}
+ C\, 2^{k}\, \|b\|_{L^\infty}\,\|\om\|_{L^q} \nonumber\\
&&\qquad \qquad + \, C\, \|\om\|_{L^q}\, \sum_{l\ge k-1} 2^{l\frac2q}\,\|\Delta_l j\|_{L^q}. \label{rhst}
\end{eqnarray}
Integrating in time, we have
\begin{eqnarray*}
\|\Delta_k j(t)\|_{L^q} \le e^{-C_1\,2^{2\beta k} t}\, \|\Delta_k j(0)\|_{L^q}
+ \int_0^t e^{-C_1\,2^{2\beta k} (t-\tau)}\, RHS(\tau)\,d\tau.
\end{eqnarray*}
For any fixed $t>0$, we take the $L^1$-norm on $[0,t]$ and apply Young's inequality to obtain
\begin{eqnarray*}
\int_0^t \|\Delta_k j(\tau)\|_{L^q}\, d\tau \le C\,2^{-2\beta k} \|\Delta_k j(0)\|_{L^q}
+ C\, 2^{-2\beta k} \int_0^t  RHS(\tau)\,d\tau.
\end{eqnarray*}
Multiplying by $2^{ks}$ and summing over $k=0,1,2,\cdots$, we obtain
\begin{eqnarray}
\|j\|_{L^1_t B^s_{q,1}} &=&\int_0^t \|\Delta_{-1} j(\tau)\|_{L^q}\,d\tau +  \sum_{k=0}^\infty 2^{ks} \int_0^t \|\Delta_k j(\tau)\|_{L^q}\, d\tau\nonumber \\
&\le& C(t) + C\, \|j_0\|_{B^{s-2\beta}_{q,1}} + \sum_{k=0}^\infty 2^{k(s-2\beta)} \int_0^t  RHS(\tau)\,d\tau, \label{jroot}
\end{eqnarray}
where we have applied the global bound $\|\Delta_{-1} j(\tau)\|_{L^q} \le \|j\|_{L^q}$. For the clarity
of presentation, we write
\begin{eqnarray*}
\sum_{k=0}^\infty 2^{k(s-2\beta)} \int_0^t  RHS(\tau)\,d\tau \equiv M_1 + M_2 + M_3 + M_4 +M_5+M_6,
\end{eqnarray*}
where, according to (\ref{rhst}),
\begin{eqnarray*}
M_1 &=& C\,\|\om\|_{L^\infty_tL^q}\, \sum_{k=0}^\infty 2^{k(s+\frac2q-2\beta)} \int_0^t \|\Delta_k j(\tau)\|_{L^q}\,d\tau,\\
M_2 &=& C\,\|\om\|_{L^\infty_tL^q}\, \sum_{k=0}^\infty \sum_{m\le k-1} 2^{(m-k)(1-s+2\beta)}\,
2^{(s+\frac{2}{q}-2\beta)m} \int_0^t \|\Delta_m j\|_{L^q}\,d\tau,\\
M_3 &=& C\,\|\om\|_{L^\infty_tL^q}\, \sum_{k=0}^\infty 2^{k(1+s-2\beta)} \sum_{l\ge k-1} 2^{(-1+\frac2q)l}\, \int_0^t \|\Delta_l j(\tau)\|_{L^q}\,d\tau,\\
M_4 &=& C\,t\,\|\om\|_{L^\infty_tL^q}\, \|j\|_{L^\infty_tL^q}\,\sum_{k=0}^\infty 2^{k(s+\frac2q-2\beta)},\\
M_5 &=& C\,\|\om\|_{L^\infty_tL^q}\,\|b\|_{L^1_tL^\infty}\,\sum_{k=0}^\infty 2^{k(s+1-2\beta)},\\
M_6 &=& C\,\|\om\|_{L^\infty_tL^q}\,\sum_{k=0}^\infty 2^{k(s-2\beta)}\sum_{l\ge k-1} 2^{l\frac2q}\,\int_0^t\|\Delta_l j(\tau)\|_{L^q}\,d\tau.
\end{eqnarray*}
Since $\frac2q-2\beta <0$, we can choose an integer $k_0>0$ such that
$$
C\,\|\om\|_{L^\infty_tL^q}\, 2^{k_0(\frac2q-2\beta)}  \le \frac1{16}.
$$
We can split the sum in $M_1$ into two parts,
\begin{eqnarray*}
M_1 &=& C\,\|\om\|_{L^\infty_tL^q}\, \sum_{k=0}^{k_0} 2^{k(s+\frac2q-2\beta)} \int_0^t \|\Delta_k j(\tau)\|_{L^q}\,d\tau\\
&& +\, C\,\|\om\|_{L^\infty_tL^q}\, \sum_{k=k_0+1}^{\infty} 2^{k(s+\frac2q-2\beta)} \int_0^t \|\Delta_k j(\tau)\|_{L^q}\,d\tau\\
&=& C(t, u_0, b_0) + \frac1{16} \|j\|_{L^1_t B^s_{q,1}},
\end{eqnarray*}
where we have used the bound
$$
\int_0^t \|\Delta_k j(\tau)\|_{L^q}\,d\tau \le \int_0^t \|j(\tau)\|_{L^q}\,d\tau \le C.
$$
To deal with $M_2$, we first realize that $1-s+2\beta>0$ and $(m-k)(1-s+2\beta)<0$,
we apply Young's inequality for series convolution to obtain
$$
M_2 \le C\,\|\om\|_{L^\infty_tL^q}\, \sum_{k=0}^\infty 2^{k(s+\frac2q-2\beta)} \int_0^t \|\Delta_k j(\tau)\|_{L^q}\,d\tau,
$$
which obeys the same bound as $M_1$, namely
$$
M_2 \le C(t, u_0, b_0) + \frac1{16} \|j\|_{L^1_t B^s_{q,1}}.
$$
To bound $M_3$, we exchange the order of two sums to get
\begin{eqnarray*}
M_3 &=& C\, \|\om\|_{L^\infty_tL^q}\, \sum_{l=-1}^\infty 2^{(-1+\frac2q)l}\,
\int_0^t \|\Delta_l j(\tau)\|_{L^q}\,d\tau \sum_{k=0}^{l+1} 2^{k(1+s-2\beta)}.
\end{eqnarray*}
Since $1+s-2\beta<0$, we have
$$
\sum_{k=0}^{l+1} 2^{k(1+s-2\beta)} \le C
$$
and thus
\begin{eqnarray*}
M_3 &=& C\, \|\om\|_{L^\infty_tL^q}\, \sum_{l=-1}^\infty 2^{(-1+\frac2q-s)l}\,
2^{ls} \int_0^t \|\Delta_l j(\tau)\|_{L^q}\,d\tau.
\end{eqnarray*}
Noticing  that $-1+\frac2q-s<0$, we take a positive integer $l_0$ such that
$$
C\, \|\om\|_{L^\infty_tL^q}\, 2^{(-1+\frac2q-s)l_0} < \frac1{16}.
$$
Then, $M_3$ is bounded by
$$
M_3 \le C(t, u_0, b_0) + \frac1{16} \|j\|_{L^1_t B^s_{q,1}}.
$$
Since $s+1-2\beta<0$, we clearly have
$$
M_4+M_5 \le C\,t\,\|\om\|_{L^\infty_tL^q}\, \|j\|_{L^\infty_tL^q}\, + C\,\|\om\|_{L^\infty_tL^q}\,\|b\|_{L^1_tL^\infty}.
$$
$M_6$ can be bounded in a similar fashion as $M_3$ and we have,  for $\frac2q<s$,
$$
M_6 \le C(t, u_0, b_0) + \frac1{16} \|j\|_{L^1_t B^s_{q,1}}.
$$
Inserting the estimates above in (\ref{jroot}), we have
\begin{eqnarray*}
\|j\|_{L^1_t B^s_{q,1}} \le C(t, u_0, b_0) + C\, \|j_0\|_{B^{s-2\beta}_{q,1}} +  \frac14\,\|j\|_{L^1_t B^s_{q,1}}.
\end{eqnarray*}
This completes the proof of Proposition \ref{Bsq1}.
\end{proof}

\vskip .3in
\section{Higher regularity through an iterative process and proof of Theorem \ref{main}}

\vskip .1in
This section explores some consequences of Proposition \ref{Bsq1}. In particular, we
obtain global bounds for $\nabla j$ in $L^1_tL^\infty_x$ and $\om$ in $L^\infty_{t,x}$, which are sufficient for the
proof of Theorem \ref{main}.  We now state the proposition for high regularity.

\begin{prop}\label{high}
Assume $(u_0, b_0)$ satisfies the conditions stated in Theorem \ref{main}. Let $(u, b)$ be the
corresponding solution of (\ref{MHDb}) with $\beta>1$. Then, for any $T>0$ and  $t\le T$, there exists a constant
$C=C(T, u_0, b_0)$ such that
$$
\int_0^t \|\nabla j(\tau)\|_{L^\infty}\,d\tau \le C, \qquad \|\om(t)\|_{L^\infty} \le C.
$$
\end{prop}

The first step is the following integrability result, as a special consequence of Proposition \ref{Bsq1}.
\begin{prop}\label{integ}
Assume $(u_0, b_0)$ satisfies the conditions stated in Theorem \ref{main}. Let $(u, b)$ be the
corresponding solution of (\ref{MHDb}) with $\beta>1$. Then, for any $r$ satisfying
\begin{equation}\label{rcon}
2\le r \le \infty \quad\mbox{if\, $\beta >\frac43$}, \quad\mbox{and}\quad  2\le r<\frac{2}{4-3\beta} \quad \mbox{if \,$\beta\le\frac43$},
\end{equation}
and, for any $T>0$ and  $t\le T$, there exists a constant
$C=C(T, u_0, b_0)$ such that
$$
\int_0^t \|\nabla j(\tau)\|_{L^r}\,d\tau \le C.
$$
\end{prop}

We remark that the range for $r$, namely (\ref{rcon}) is bigger than the one for $q$ in (\ref{qran}).
An immediate consequence is the global bound for $\|\om\|_{L^r}$ and $\|j\|_{L^r}$, as explained in the
proof of Proposition \ref{high}.

\begin{proof}[Proof of Proposition \ref{integ}]
By Bernstein's inequality,
\begin{eqnarray*}
\|\nabla j\|_{L^r} &\le& \sum_{k=-1}^\infty \|\Delta_k \nabla j\|_{L^r}\\
&\le& \sum_{k=-1}^\infty 2^{k(1+ \frac2q-\frac2r-s)}  2^{ks} \|\Delta_k j\|_{L^q}.
\end{eqnarray*}
In the case when $\beta>\frac43$, we can choose $q$ and $s$ satisfying (\ref{qscon}), say $q=3$ and $s=\frac53$, such that
$$
1+ \frac2q-s \le 0,
$$
and consequently, for any $r\in [2, \infty]$,
$$
\|\nabla j\|_{L^r} \le \sum_{k=-1}^\infty 2^{ks} \|\Delta_k j\|_{L^q} \equiv \|j\|_{B^s_{q,1}}.
$$
In the case when $\beta\le \frac43$, we can choose $q$ and $s$ satisfying (\ref{qscon}) and $r$ satisfying (\ref{rcon})
such that
$$
1+ \frac2q-\frac2r-s \le 0
$$
and again
$$
\|\nabla j\|_{L^r} \le \|j\|_{B^s_{q,1}}.
$$
Proposition \ref{integ} then follows from Proposition \ref{Bsq1}.
\end{proof}

We now prove Proposition \ref{high}.
\begin{proof}[Proof of Proposition \ref{high}]
Proposition \ref{integ} allows us to obtain a global bound for $\|\om\|_{L^r}$. In fact, it
follows from the vorticity equation that
$$
\frac1r \frac{d}{dt} \|\om\|_{L^r}^r \le \|b\|_{L^\infty} \|\nabla j\|_{L^r} \, \|\om\|_{L^r}^{r-1}.
$$
Due to the Sobolev inequality, for $q>2$,
$$
\|b\|_{L^\infty} \le C \|b\|_{L^2}^{\frac{q-2}{2(q-1)}}\, \|j\|_{L^q}^{\frac{q}{2q-2}}
$$
and the time integrability of $\|\nabla j\|_{L^r}$ from Proposition \ref{integ}, we obtain the global bound
$$
\|\om(t)\|_{L^r} \le \|\om_0\|_{L^r} + \|b\|_{L^\infty_t L^\infty_x} \int_0^t \|\nabla j\|_{L^r}\,d\tau < \infty.
$$
By going through the proof of Proposition \ref{Lq} with $q$ replaced by $r$, we can show that
$$
\|j(t)\|_{L^r} \le C.
$$
As a consequence of the global bounds for $\|\om(t)\|_{L^r}$ and $\|j(t)\|_{L^r}$, we can prove
Proposition \ref{Bsq1} again with $q$ replaced by $r$. This iterative process allows us to establish
the global bound
$$
\|j\|_{L^1_t B^s_{r,1}} <\infty
$$
for any $r\in [2, \infty)$ and $\frac2r<s<2\beta-1$. As a special consequence, we have, for any $t>0$,
\begin{eqnarray} \label{nj}
\int_0^t \|\nabla j\|_{L^\infty}\,d\tau <\infty.
\end{eqnarray}
In fact,
\begin{eqnarray*}
\|\nabla j\|_{L^\infty} &\le& \sum_{k=-1}^\infty \|\Delta_k \nabla j\|_{L^\infty}
\le \sum_{k=-1}^\infty 2^{k(1+\frac2r)}\, \|\Delta_k j\|_{L^r}\\
&=& \sum_{k=-1}^\infty 2^{k(1+\frac2r-s)}\, 2^{ks}\|\Delta_k j\|_{L^r}
= \sum_{k=-1}^\infty 2^{ks}\|\Delta_k j\|_{L^r} \equiv \|j\|_{B^s_{r,1}},
\end{eqnarray*}
where we have choose $r$ large and $s<2\beta-1$ such that
$$
1+\frac2r-s \le 0.
$$
The global bound in (\ref{nj}) further allows us to show that
\begin{eqnarray} \label{omb}
\|\om\|_{L^\infty} <\infty,
\end{eqnarray}
which follows from the inequality
$$
\|\om(t)\|_{L^\infty} \le \|\om_0\|_{L^\infty} + \|b\|_{L^\infty_t L^\infty} \int_0^t \|\nabla j\|_{L^\infty}\,d\tau.
$$
This completes the proof of Proposition \ref{high}.
\end{proof}

\vskip .1in
Finally we prove Theorem \ref{main}.
\begin{proof}[Proof of Theorem \ref{main}] The proof of Theorem \ref{main} is divided into two main steps. The first step is to construct
a local (in time) solution while the second step extends the local solution into a global one by making use of the
global {\it a priori} bounds obtained in Proposition \ref{high}. The construction of a local solution is quite standard
and is thus omitted here. The global bounds in Proposition \ref{high} are sufficient in proving the global bound
$$
\|(u,b)\|_{H^s} < \infty.
$$
This completes the proof of Theorem \ref{main}.
\end{proof}

\vskip .4in
\section*{Acknowledgements}
Cao was partially supported by NSF grant DMS-1109022. Wu was partially supported by NSF grant DMS-1209153. Wu also
acknowledges the support of Henan Polytechnic University.
Yuan was partially supported by the National Natural
Science Foundation of China (No. 11071057) and by Innovation Scientists and Technicians Troop
Construction Projects of Henan Province (No. 104100510015).

\end{document}